\documentclass[10pt]{amsart}
\usepackage{amssymb}
\usepackage{bm}
\usepackage{graphicx}
\usepackage[centertags]{amsmath}
\usepackage{amsfonts}
\usepackage{amsthm}
\newtheorem{thm}{Theorem}
\newtheorem{cor}[thm]{Corollary}
\newtheorem{lem}[thm]{Lemma}

\newtheorem{prop}[thm]{Proposition}

\theoremstyle{definition}
\newtheorem{rem}{Remark}

\newcommand{\rr}{\mathbb{R}}
\newcommand{\nn}{\mathbb{N}}
\newcommand{\con}{\smallfrown}
\newcommand{\ee}{\varepsilon}
\newcommand{\sbs}{\mathrm{SB}}
\newcommand{\sd}{\mathrm{SD}}
\newcommand{\refl}{\mathrm{REFL}}
\newcommand{\diam}{\mathrm{diam}}
\newcommand{\sz}{\mathrm{Sz}}
\newcommand{\sss}{\mathcal{S}}
\newcommand{\ccc}{\mathcal{C}}

\newcommand{\PB}{\mathbf{\Pi}}

\begin{document}

\title{Definability under duality}
\author{Pandelis Dodos}
\address{Universit\'{e} Pierre et Marie Curie - Paris 6, Equipe d' Analyse
Fonctionnelle, Bo\^{i}te 186, 4 place Jussieu, 75252 Paris Cedex 05, France.}
\email{pdodos@math.ntua.gr}

\footnotetext[1]{2000 \textit{Mathematics Subject Classification}: 03E15, 46B10.}
\footnotetext[2]{\textit{Key words}: Banach spaces with separable dual, analytic
classes of Banach spaces, Szlenk index, selection.}

\maketitle


\begin{abstract}
It is shown that if $A$ is an analytic class of separable Banach spaces with
separable dual, then the set $A^*=\{ Y:\exists X\in A \text{ with } Y\cong X^*\}$
is analytic. The corresponding result for pre-duals is false.
\end{abstract}


\section{Introduction}

\noindent \textbf{(A)} All separable Banach spaces can be realized, up to isometry, as subspaces of
$C(2^\nn)$. Denoting by $\sbs$ the set of all closed linear subspaces of $C(2^\nn)$ and endowing $\sbs$
with the relative Effros-Borel structure, the set $\sbs$ becomes the standard Borel space of all
separable Banach spaces (see \cite{AD}, \cite{AGR}, \cite{Bos} and \cite{Kechris}). By identifying
any class of separable Banach spaces with a subset of $\sbs$, the space $\sbs$ provides the
appropriate frame for studying structural properties of classes of Banach spaces.
This identification is ultimately related to universality problems in Banach Space Theory. This is
justified by a number of results (\cite{AD}, \cite{DF}, \cite{D} and \cite{DLo})
of which the following one, taken from \cite{DF}, is a sample.
\medskip

\noindent \textit{If $A$ is an analytic subset of $\sbs$ such that every $X\in A$ is 
reflexive, then there exists a reflexive Banach space $Y$, with a Schauder basis, 
that contains isomorphic copies of every $X\in A$.}
\medskip

\noindent To see how such a result is used, let us consider the set $\mathrm{UC}$ consisting 
of all $X\in\sbs$ which are uniformly convex. It is a classical fact (see \cite{LT}) that 
$\mathrm{UC}$ contains only reflexive spaces. Moreover, it is easily checked that
$\mathrm{UC}$ is a Borel subset of $\sbs$. Applying the above result, we recover a recent
result of E. Odell and Th. Schlumprecht \cite{OS} asserting the existence of a separable
reflexive space $R$ containing an isomorphic copy of every separable uniformly convex Banach space.
The problem of the existence of such a space was posed by Jean Bourgain \cite{Bou2}.
\medskip

\noindent \textbf{(B)} As we have already indicated, in applications one has to decide whether 
a given class of separable Banach spaces is analytic or not. Sometimes this is straightforward
to check invoking, simply, the definition of the class. There are classes, however, which are
defined implicitly using a certain Banach space operation. In these cases, usually, deeper
arguments are involved.

This note is concerned with the question whether analyticity is preserved under duality, a very basic
operation encountered in Banach Space Theory. Precisely, the following two questions are naturally
asked in such a context.
\begin{enumerate}
\item[\textbf{(Q1)}] If $A$ is an analytic class of separable dual Banach spaces, then
is the set $A_*=\{ X\in\sbs: \exists Y\in A \text{ with } X^*\cong Y\}$ analytic?
\item[\textbf{(Q2)}] If $A$ is an analytic class of separable Banach spaces with
separable dual, then is the set $A^*=\{Y\in\sbs: \exists X\in A \text{ with } Y\cong X^*\}$
analytic?
\end{enumerate}
Question (Q1) has a negative answer and a counterexample is the set $A=\{ Y\in\sbs:
Y\cong \ell_1\}$, i.e. the isomorphic class of $\ell_1$ (a more detailed explanation
will be given later on). However, for question (Q2) we do have a positive result.
\begin{thm}
\label{t1} Let $A$ be an analytic class of separable Banach spaces with separable dual.
Then the set $A^*=\{ Y\in\sbs: \exists X\in A \text{ with } Y\cong X^*\}$ is analytic.
\end{thm}
The proof of Theorem \ref{t1} is based on a selection result which is, perhaps, of
independent interest. To state it, let $H=[-1,1]^\nn$ equipped with the product topology.
That is, $H$ is the closed unit ball of $\ell_\infty$ with the weak* topology. A subset
$S$ of $H$ will be called norm separable if it is separable with respect to the metric induced
by the supremum norm $\|\cdot\|_\infty$. The selection result we need is the following.
\begin{prop}
\label{t2} Let $Z$ be a standard Borel space and
$A\subseteq Z\times H$ Borel such that the following hold.
\begin{enumerate}
\item[(1)] For every $z\in Z$, the section $A_z$ is non-empty and compact.
\item[(2)] For every $z\in Z$, the section $A_z$ is norm separable.
\end{enumerate}
Then there exists a sequence $(f_n)$ of Borel selectors of $A$ such that
for all $z\in Z$ the sequence $\big(f_n(z)\big)$ is norm dense in $A_z$.
\end{prop}
As usual, a map $f:Z\to H$ is said to be a \textit{Borel selector} of $A$ if $f$ is a Borel
map such that $\big(z,f(z)\big)\in A$ for every $z\in Z$. The proof of Proposition \ref{t2}
is based on a Szlenk type index defined on all norm-separable compact subsets of $H$.
Actually, what we use is the fact that this index has nice definability properties
(it is a co-analytic rank) and it satisfies boundedness. We should remark that the use
of boundedness in selection theorems is common in descriptive set theory (it is used,
for instance, in the proof of the strategic uniformization theorem -- see
\cite[Theorem 35.32]{Kechris}). We also notice that the transfinite manipulations
made in the proof of Proposition \ref{t2} are similar to the ones in the selection
theorems of J. Jayne and A. Rogers \cite{JR} as well as of N. Ghoussoub, B. Maurey
and W. Schachermayer \cite{GMS}. We point out, however, that the crucial definability
considerations in the proof of Proposition \ref{t2} do not appear in \cite{JR} and \cite{GMS}.

\subsection{Notation} We let $\nn=\{0,1,2,...\}$. For every Polish space $X$, by $K(X)$ we denote
the set of all compact subsets of $X$ (the empty set is included). We equip $K(X)$ with the
Vietoris topology $\tau_V$, i.e. the one generated by the sets
\[ \{K\in K(X): K\cap U\neq \varnothing\} \ \text{ and } \ \{K\in K(X): K\subseteq U\}\]
where $U$ ranges over all non-empty open subsets of $X$. It is well-known
(see \cite{Kechris}) that the space $(K(X),\tau_V)$ is Polish.
A map $D:K(X)\to K(X)$ is said to be a \textit{derivative} on $K(X)$ provided that
$D(K)\subseteq K$ and $D(K_1)\subseteq D(K_2)$ if $K_1\subseteq K_2$. For every
$K\in K(X)$, by transfinite recursion one defines the \textit{iterated derivatives}
$D^{(\xi)}(K)$ of $K$ by the rule
\[ D^{(0)}(K)=K, \ D^{(\xi+1)}(K)=D\big( D^{(\xi)}(K)\big) \text{ and }
D^{(\lambda)}=\bigcap_{\xi<\lambda} D^{(\xi)}(K) \text{ if } \lambda \text{ is limit}.\]
The \textit{$D$-rank} of $K$ is the least ordinal $\xi$ for which $D^{(\xi)}(K)= D^{(\xi+1)}(K)$.
It is denoted by $|K|_D$. Moreover, we set $D^{(\infty)}(K)=D^{|K|_D}(K)$. If $X,Y$
are sets and $A\subseteq X\times Y$, then for every $x\in X$ by $A_x$ we denote the
section of $A$ at $x$, i.e. the set $\{y:(x,y)\in A\}$. All the other pieces of
notation we use are standard (see for instance \cite{Kechris} or \cite{LT}).

\subsection{The counterexample to question (Q1)} We have already
mention that the counterexample is the isomorphic class of $\ell_1$, that is
the set $A=\{Y: Y\cong \ell_1\}$. As the equivalence relation of isomorphism
$\cong$ is analytic in $\sbs\times\sbs$ (see \cite{Bos}), the set $A$ is analytic.
We will show that the set $A_*=\{ X: \exists Y\in A \text{ with } X^*\cong Y\}=
\{X: X^*\cong\ell_1\}$ is not analytic. The argument below goes back to the
fundamental work of J. Bourgain on $C(K)$ spaces, with $K$ countable compact
(see \cite{Bou1}). Specifically, there exists a Borel map $\Phi:K(2^\nn)\to\sbs$
such that for all $K\in K(2^\nn)$ the space $\Phi(K)$ is isomorphic to $C(K)$
(see \cite{Kechris}, page 263). Denote by $K_\omega(2^\nn)$ the set of all
countable compact subsets of $2^\nn$. It follows that
\[ K\in K_\omega(2^\nn) \Leftrightarrow C(K)^*\cong \ell_1 \Leftrightarrow
\Phi(K)\in A_*. \]
Hence $K_\omega(2^\nn)=\Phi^{-1}(A_*)$. By a classical result of Hurewicz
(see \cite[Theorem 27.5]{Kechris}), the set $K_\omega(2^\nn)$ is
co-analytic non-Borel and so the set $A_*$ is not analytic (for if not, we would
have that $K_\omega(2^\nn)$ is analytic). In descriptive set-theoretic terms,
the above argument shows that the set $A_*$ is Borel $\PB^1_1$-hard.


\section{Proof of Proposition \ref{t2}}

In what follows, by $H$ we shall denote the set $[-1,1]^\nn$ equipped with the product topology.
Let us recall the following well-known topological lemma (see \cite{GM} or \cite{Ro} and the
references therein). For the sake of completeness we include a proof.
\begin{lem}
\label{l1} Let $K\subseteq H$ non-empty compact. If $K$ is norm separable, then
for every $\ee>0$ there exists $U\subseteq H$ open such that $K\cap U\neq \varnothing$
and $\|\cdot\|_{\infty}-\diam(K\cap U)\leq\ee$.
\end{lem}
\begin{proof}
We fix a compatible metric $\rho$ for $H$ with $\rho-\mathrm{diam}(H)\leq 1$ (notice that such a metric
$\rho$ is necessarily complete). Assume, towards a contradiction, that the lemma is false. Hence,
we can construct a family $(V_t)$ $(t\in 2^{<\nn})$ of non-empty relatively open subsets of $K$
such that the following are satisfied.
\begin{enumerate}
\item[(a)] For every $t\in 2^{<\nn}$ we have
$\overline{V}_{t^{\con}0}\cap \overline{V}_{t^{\con}1}=\varnothing$,
$\big(\overline{V}_{t^{\con}0}\cup \overline{V}_{t^{\con}1}\big)\subseteq V_t$ and
$\rho-\diam(V_t)\leq 2^{-|t|}$.
\item[(b)] For every $n\in\nn$ with $n\geq 1$, every $t,s\in 2^n$ with $t\neq s$ and every pair
$(f,g)\in V_t\times V_s$ we have $\|f-g\|_\infty>\ee$.
\end{enumerate}
We set $P=\bigcup_{\sigma\in 2^\nn} \bigcap_{n\in\nn} V_{\sigma|n}$.
By (a) above, $P$ is a perfect subset of $K$. On the other hand, by (b), we see that
$\|f-g\|_\infty>\ee$ for every  $f,g\in P$ with $f\neq g$. That is, $K$ is not norm separable,
a contradiction. The proof is completed.
\end{proof}
Lemma \ref{l1} suggests a canonical derivative operation on compact subsets
of $H$, similar to the derivative operation appearing in W. Szlenk's analysis
of separable dual spaces \cite{Sz}. Actually, our interest on it stems from
the fact that it has the right definability properties.

To define this derivative, let $(U_n)$ be an enumeration of a countable basis of $H$
(we will assume that every $U_n$ is non-empty). This basis will be fixed.
Let $\ee>0$ be arbitrary. Define $D_{n,\ee}:K(H)\to K(H)$ by
\[ D_{n,\ee}(K)= \left\{ \begin{array}
{r@{\quad:\quad}l} K\setminus U_n & \text{if } K\cap U_n\neq \varnothing \text{ and }
\|\cdot\|_{\infty}-\diam(K\cap U_n)\leq\ee, \\
K & \mbox{otherwise.} \end{array} \right. \]
Notice that $D_{n,\ee}$ is a derivative on $K(H)$. Now define $\mathbb{D}_\ee:K(H)\to K(H)$
by $\mathbb{D}_\ee(K)=\bigcap_n D_{n,\ee}(K)$. That is
\[ \mathbb{D}_\ee(K)=K \setminus \bigcup\{ U_n: K\cap U_n\neq\varnothing \text{ and }
\|\cdot\|_{\infty}-\diam(K\cap U_n)\leq\ee\} .\]
Clearly $\mathbb{D}_\ee$ is derivative on $K(H)$ too.
\begin{lem}
\label{l2} Let $\ee>0$. Then the following hold.
\begin{enumerate}
\item[(i)] For every $n\in\nn$, the map $D_{n,\ee}$ is Borel.
\item[(ii)] The map $\mathbb{D}_{\ee}$ is a Borel derivative.
\end{enumerate}
\end{lem}
\begin{proof}
(i) Fix $n\in\nn$. Let
\[ A_n=\{ K\in K(H): K\cap U_n\neq\varnothing \text{ and }
\|\cdot\|_{\infty}-\diam(K\cap U_n)\leq\ee\}.\]
Then $A_n$ is Borel (actually it is the complement of a $K_\sigma$ set) in $K(H)$, as
\begin{eqnarray*}
K\notin A_n & \Leftrightarrow & (K\cap U_n=\varnothing) \text { or } \big(\exists f,g\in K\cap U_n\\
& & \exists l\in\nn \ \exists m\in\nn \text{ with } |f(l)-g(l)|\geq \ee+\frac{1}{m+1}\big).
\end{eqnarray*}
Now observe that $D_{n,\ee}(K)=K$ if $K\notin A_n$ and $D_{n,\ee}(K)=K\setminus U_n$ if
$K\in A_n$. This easily implies that $D_{n,\ee}$ is Borel. \\
(ii) Consider the map $F:K(H)^\nn\to K(H)^\nn$ defined by
\[ F\big( (K_n)\big)=\big( D_{n,\ee}(K_n)\big). \]
By part (i), we have that $F$ is Borel. Moreover, by \cite[Lemma 34.11]{Kechris},
the map $\bigcap: K(H)^\nn\to K(H)$ defined by $\bigcap\big((K_n)\big)=\bigcap_n K_n$
is Borel too. Finally, let $I:K(H)\to K(H)^\nn$ be defined by $I(K)=(K_n)$ with
$K_n=K$ for every $n$. Clearly $I$ is continuous. As $\mathbb{D}_{\ee}(K)=
\bigcap\big(F(I(K))\big)$, the result follows.
\end{proof}
We will need the following well-known result concerning sets in product spaces
with compact sections (see \cite[Theorem 28.8]{Kechris}).
\begin{thm}
\label{tc} Let $Z$ be a standard Borel space, $H$ a Polish space and $A\subseteq
Z\times H$ with compact sections. Let $\Phi_A:Z\to K(H)$ be defined by $\Phi_A(z)=A_z$
for all $z\in Z$. Then $A$ is Borel in $Z\times H$ if and only if $\Phi_A$
is a Borel map.
\end{thm}
Now let $B\subseteq H$ and $\ee>0$. We say that a subset $S$ of $B$ is
\textit{norm $\ee$-dense} in $B$ if for every $g\in B$ there exists $f\in S$
with $\|f-g\|_{\infty}\leq\ee$.
\begin{lem}
\label{l3} Let $Z$ and $A$ be as in the statement of Proposition \ref{t2}.
Let also $\ee>0$ and $\tilde{A}\subseteq Z\times H$ Borel with
$\tilde{A}\subseteq A$ and such that for every $z\in Z$ the section
$\tilde{A}_z$ is a (possibly empty) compact set. Then there exists a
sequence $(f_n)$ of Borel selectors of $A$ such that for all $z\in Z$,
if the section $\tilde{A}_z$ is non-empty, then the set $\{ f_n(z): f_n(z)\in
\tilde{A}_z\setminus\mathbb{D}_\ee(\tilde{A}_z)\}$ is non-empty and norm
$\ee$-dense in $\tilde{A}_z\setminus \mathbb{D}_\ee(\tilde{A}_z)$.
\end{lem}
\begin{proof}
Let $n\in\nn$. By Theorem \ref{tc}, the map $\Phi_{\tilde{A}}$ is Borel. Let $Z_n=\{ z\in Z:
\tilde{A}_z\cap U_n\neq\varnothing \text{ and } \|\cdot\|_{\infty}-\diam(\tilde{A}_z\cap U_n)
\leq\ee\}$. Then $Z_n$ is Borel in $Z$. To see this, notice that $Z_n=\Phi^{-1}_{\tilde{A}}(A_n)$,
where $A_n$ is defined in the proof of Lemma \ref{l2}(i). Now let $\tilde{A}_n\subseteq Z\times H$
be defined by the rule
\begin{eqnarray*}
(z,f)\in \tilde{A}_n & \Leftrightarrow & \big( z\in Z_n \text{ and } f\in U_n \text{ and }
(z,f)\in \tilde{A} \big) \text{ or } \\ & & \big( z\notin Z_n \text{ and } (z,f)\in A \big).
\end{eqnarray*}
It is easy to see that for every $n\in\nn$, the set $\tilde{A}_n$ is a
Borel set with non-empty $\sigma$-compact sections. By the Arsenin-Kunugui
theorem (see \cite[Theorem 35.46]{Kechris}), there exists a Borel map
$f_n:Z\to H$ such that $\big(z,f_n(z)\big)\in \tilde{A}_n$ for all
$z\in Z$. We claim that the sequence $(f_n)$ is the desired one. Clearly
it is a sequence of Borel selectors of $A$. What remains is to check that it has the desired
property. So, let $z\in Z$ such that $\tilde{A}_z$ is non-empty and let $f\in \tilde{A}_z\setminus
\mathbb{D}_\ee(\tilde{A}_z)$. It follows readily by the definition of $\mathbb{D}_\ee$ that there
exists $n_0\in\nn$ such that $z\in Z_{n_0}$ and $(z,f)\in\tilde{A}_{n_0}$. The definition
of $\tilde{A}_{n_0}$ yields that the set $\{h:(z,h)\in\tilde{A}_{n_0}\}$ has norm diameter
less or equal to $\ee$. As $(z,f_{n_0}(z))\in\tilde{A}_{n_0}$, we conclude that
$\|f-f_{n_0}(z)\|_{\infty}\leq\ee$ and the proof is completed.
\end{proof}
Before we proceed to the proof of Proposition \ref{t2}, we need the following facts
about the derivative operation $\mathbb{D}_\ee$ described above. By Lemma \ref{l2},
the map $\mathbb{D}_\ee$ is a Borel derivative on $K(H)$. It follows
by \cite[Theorem 34.10]{Kechris} that the set
\[ \Omega_{\mathbb{D}_\ee}=\{ K\in K(H): \mathbb{D}_\ee^{(\infty)}(K)=\varnothing \}\]
is co-analytic and that the map $K\to |K|_{\mathbb{D}_\ee}$ is a co-analytic rank on
$\Omega_{\mathbb{D}_\ee}$ (a $\PB^1_1$-rank in the technical logical jargon -- see
\cite{Kechris} for the definition and the properties of co-analytic ranks). We are
particulary interested in the following important property which
is shared by all co-analytic ranks (see \cite[Theorem 35.22]{Kechris}). If
$S$ is an analytic subset of $\Omega_{\mathbb{D}_\ee}$, then
\[ \sup\{ |K|_{\mathbb{D}_\ee}: K\in S\}<\omega_1 \]
(this property is known as boundedness). We are now ready to give the proof of
Proposition \ref{t2}.
\begin{proof}[Proof of Proposition \ref{t2}]
Let $A\subseteq Z\times H$ be as in the statement of the proposition. By Theorem \ref{tc},
the map $\Phi_A:Z\to K(H)$ defined by $\Phi_A(z)=A_z$ is Borel, and so, the set
$\{ A_z: z\in Z\}$ is an analytic subset of $K(H)$.

Now, let $\ee>0$ be arbitrary and consider the derivative operation $\mathbb{D}_\ee$.
By our assumptions on $A$ and by Lemma \ref{l1}, we see that for every
$z\in Z$ and every $\xi<\omega_1$ if $\mathbb{D}_\ee^{(\xi)}(A_z)\neq\varnothing$,
then $\mathbb{D}_\ee^{(\xi+1)}(A_z)\subsetneqq\mathbb{D}_\ee^{(\xi)}(A_z)$.
It follows that the transfinite sequence $\big( \mathbb{D}_\ee^{(\xi)}(A_z)\big)$
$(\xi<\omega_1)$ must be stabilized at $\varnothing$, and so,
$\{ A_z:z\in Z\}\subseteq\Omega_{\mathbb{D}_\ee}$. Hence, by boundedness, we get that
\[ \sup\big\{ |A_z|_{\mathbb{D}_\ee}: z\in Z\big\}=\xi_\ee<\omega_1.\]
For every $\xi<\xi_\ee$ we define recursively $A^\xi\subseteq Z\times
H$ as follows. Let $A^0=A$. If $\xi=\zeta+1$ is a successor
ordinal, define $A^\xi$ by
\[ (z,f)\in A^\xi \Leftrightarrow f\in \mathbb{D}_\ee\big( (A^\zeta)_z \big) \]
where $(A^\zeta)_z$ is the section $\{f:(z,f)\in A^\zeta\}$ of $A^\zeta$. If $\xi$ is limit,
then let
\[ (z,f)\in A^\xi \Leftrightarrow (z,f)\in\bigcap_{\zeta<\xi} A^\zeta. \]
\noindent \textsc{Claim.} \textit{The following hold.
\begin{enumerate}
\item[(1)] For every $\xi<\xi_\ee$, the set $A^\xi$ is a Borel
subset of $A$ with compact sections.
\item[(2)] For every $(z,f)\in Z\times H$ with $(z,f)\in A$,
there exists a unique ordinal $\xi<\xi_\ee$ such that $(z,f)\in
A^{\xi}\setminus A^{\xi+1}$, equivalently
$f\in(A^{\xi})_z\setminus \mathbb{D}_\ee\big( (A^\xi)_z\big)$.
\end{enumerate} }
\medskip

\noindent \textit{Proof of the claim.} (1) By induction on all
ordinals less than $\xi_\ee$. For $\xi=0$ it is straightforward.
If $\xi=\zeta+1$ is a successor ordinal, then by our inductive
hypothesis and Theorem \ref{tc}, the map  $z\mapsto (A^\zeta)_z$
is Borel. By Lemma \ref{l2}(ii), the map $z\mapsto
\mathbb{D}_\ee\big( (A^\zeta)_z\big)$ is Borel too. By the
definition of $A^{\xi}=A^{\zeta+1}$ and invoking Theorem \ref{tc}
once more, we conclude that $A^\xi$ is a Borel subset of $A$ with
compact sections. If $\xi$ is limit, then this is an
immediate consequence of our inductive hypothesis and the definition of $A^\xi$.\\
(2) For every $z\in Z$, let $\xi_z=|A_z|_{\mathbb{D}_\ee}\leq\xi_\ee$.
Notice that $A_z$ is partitioned into the disjoint sets $\mathbb{D}_\ee^{(\xi)}(A_z)
\setminus \mathbb{D}^{(\xi+1)}_\ee(A_z)$ with $\xi<\xi_z$. By transfinite induction,
one easily shows that $(A^{\xi})_z=\mathbb{D}_\ee^{(\xi)}(A_z)$ for every $\xi<\xi_z$.
It follows that
\[ \mathbb{D}_\ee^{(\xi)}(A_z)\setminus \mathbb{D}^{(\xi+1)}_\ee(A_z)=(A^\xi)_z\setminus
(A^{\xi+1})_z=(A^\xi)_z\setminus \mathbb{D}_\ee\big( (A^\xi)_z\big).\]
The claim is proved. \hfill $\lozenge$
\medskip

\noindent By part (1) of the claim, for every $\xi<\xi_\ee$ we may
apply Lemma \ref{l3} for the set $A^\xi$. Therefore, we get for all
$\xi<\xi_\ee$ a sequence $(f^\xi_n)$ of Borel selectors of $A$
as described in Lemma \ref{l3}. Enumerate the sequence
$(f^\xi_n)$ $(\xi<\xi_\ee, n\in\nn)$ in a single sequence, say as
$(f_n)$. Clearly the sequence $(f_n)$ is a sequence of Borel
selectors of $A$. Moreover, by part (2) of the above claim and the
properties of the sequence obtained by Lemma \ref{l3}, we see that
for all $z\in Z$ the set $\{f_n(z):n\in\nn\}$ is norm $\ee$-dense
in $A_z$. Applying the above for $\ee=(m+1)^{-1}$ with $m\in\nn$,
we get the result.
\end{proof}


\section{Proof of Theorem \ref{t1}}

Before we embark into the proof, we need to discuss some standard
facts (see \cite{Kechris}, page 264). First of all we notice that
an application of the Kuratowski--Ryll-Nardzewski selection Theorem
(see \cite[Theorem 12.13]{Kechris}) yields that there exists a sequence
$d_n:\mathrm{SB}\to C(2^\nn)$ $(n\in\nn)$ of Borel functions such that
for every $X\in\sbs$, the sequence $\big(d_n(X)\big)$
is dense in $X$ and closed under rational linear combinations.

Using this, for every $X\in\sbs$ we can identify the closed unit ball
$B_1(X^*)$ of $X^*$ with a compact subset $K_{X^*}$ of $H=[-1,1]^\nn$.
In particular, we view every element $x^*\in B_1(X^*)$ as an element $f\in H$
by identifying it with the sequence $n\mapsto \frac{x^*(d_n(X))}
{\|d_n(X)\|}$ (if $d_n(X)=0$, then we define this ratio to be 0).
There are two crucial properties established with this identification.
\begin{enumerate}
\item[(P1)] The set $D\subseteq\sbs\times H$ defined by
\[ (X,f)\in D \Leftrightarrow f\in K_{X^*}\]
is Borel. Indeed, notice that
\begin{eqnarray*}
(X,f)\in\mathrm{D} & \Leftrightarrow & \forall n,m,k\in\nn \ \forall p,q\in\mathbb{Q} \text{ we have }\\
& & \big[ p\cdot d_n(X)+q\cdot d_m(X)=d_k(X) \Rightarrow\\
& & \ p\cdot \|d_n(X)\|\cdot f(n)+q\cdot \|d_m(X)\|\cdot f(m)=\|d_k(X)\|\cdot f(k) \big].
\end{eqnarray*}

\item[(P2)] If $f_0,...,f_k\in K_{X^*}$ and $x_0^*,...,x^*_k$ are the
corresponding elements of $B_1(X^*)$, then for every $a_0,...,a_k\in\rr$ we have
\begin{eqnarray*}
\Big\| \sum_{i=0}^k a_ix_i^*\Big\|_{X^*} & = & \sup\Big\{ \Big|
\sum_{i=0}^k a_i\frac{x_i^*(d_n(X))}{\|d_n(X)\|}\Big|: d_n(X)\neq 0\Big\}\\
& = & \sup\Big\{ \Big| \sum_{i=0}^k a_i f_i(n)\Big|: n\in\nn\Big\}
= \Big\| \sum_{i=0}^k a_i f_i \Big\|_{\infty}.
\end{eqnarray*}
In other words, this identification of $B_1(X^*)$ with $K_{X^*}$ is isometric.
\end{enumerate}
We proceed to the proof of Theorem \ref{t1}.
\begin{proof}[Proof of Theorem \ref{t1}]
Let $A$ be an analytic subset of $\sbs$ such that every $X\in A$ has separable dual.
Denote by $\sd$ the set of all $X\in\sbs$ with separable dual. It is co-analytic (see,
for instance, \cite[Theorem 33.24]{Kechris}). Hence, by Lusin's separation theorem
(see \cite[Theorem 14.7]{Kechris}), there exists $Z\subseteq\sd$ Borel
with $A\subseteq Z$. Define $G\subseteq Z\times H$ by
\[ (X,f)\in G\Leftrightarrow f\in K_{X^*}. \]
It follows by property (P1) above that $G$ is a Borel set such that for every $X\in Z$
the section $G_X$ of $G$ at $X$ is non-empty, compact and norm-separable.
We apply Proposition \ref{t2} and we get a sequence $f_n:Z\to H$ $(n\in\nn)$
of Borel selectors of $G$ such that for every $X\in Z$ the sequence
$\big( f_n(X)\big)$ is norm dense in $G_X=K_{X^*}$. Notice that, by
property (P2) above, for every $Y\in\sbs$ and every $X\in Z$ we have
\begin{eqnarray*}
Y\cong X^* & \Leftrightarrow & \exists (y_n)\in Y^\nn \ \exists k\geq 1
\text{ with } \overline{\mathrm{span}}\{y_n:n\in\nn\}=Y \\
& & \text{and } (y_n) \stackrel{k}{\sim} \big(f_n(X)\big)
\end{eqnarray*}
where as usual $(y_n) \stackrel{k}{\sim} \big(f_n(X)\big)$ if for every $m\in\nn$ and
every $a_0,...,a_m\in\rr$ we have
\[ \frac{1}{k}\Big\| \sum_{n=0}^m a_n y_n\Big\| \leq \Big\| \sum_{n=0}^m a_n f_n(X)
\Big\|_{\infty}\leq k \Big\| \sum_{n=0}^m a_n y_n\Big\|. \]
For every $k\in\nn$ with $k\geq 1$, consider the relation $E_k$ in
$C(2^\nn)^\nn\times H^\nn$ defined by
\[ \big( (y_n),(h_n)\big)\in E_k \Leftrightarrow (y_n) \stackrel{k}{\sim} (h_n).\]
Then $E_k$ is Borel as
\begin{eqnarray*}
(y_n) \stackrel{k}{\sim} (h_n) &\Leftrightarrow & \forall m \ \forall a_0,...,a_m\in
\mathbb{Q} \ \Big( \forall l \ \Big| \sum_{n=0}^m a_n h_n(l)\Big| \leq k\Big\| \sum_{n=0}^m
a_ny_n \Big\| \Big) \\
& & \text{and } \Big( \forall p \ \exists i \ \frac{1}{k}\Big\| \sum_{n=0}^m
a_ny_n \Big\| - \frac{1}{p+1} \leq \Big| \sum_{n=0}^m a_n h_n(i)\Big| \Big).
\end{eqnarray*}
The sequence $(f_n)$ consists of Borel functions, and so, the relation $I_k$ in
$C(2^\nn)^\nn\times Z$ defined by
\[ \big( (y_n),X\big)\in I_k \Leftrightarrow \big( (y_n), (f_n(X)) \big)\in E_k \]
is Borel. Finally, the relation $S$ in $\sbs\times C(2^\nn)^\nn$ defined by
\[ \big(Y,(y_n)\big)\in S\Leftrightarrow (\forall n \ y_n\in Y) \text{ and }
\overline{\mathrm{span}}\{y_n:n\in\nn\}=Y \]
is Borel (see \cite[Lemma 2.6]{Bos}). Now let $A^*=\{ Y\in\sbs:\exists X\in A
\text{ with } X^*\cong Y\}$. It follows by the above discussion that
\begin{eqnarray*}
Y\in A^* & \Leftrightarrow & \exists X\in A \ \exists (y_n)\in C(2^\nn)^\nn \
\exists k\geq 1 \text{ with } \big( Y,(y_n)\big)\in S \\
& & \text{and } \big( (y_n), X \big)\in I_k.
\end{eqnarray*}
Clearly the above formula gives an analytic definition of $A^*$, as desired.
\end{proof}


\section{Further Consequences}

The following proposition is a second application
of Proposition \ref{t2}. It implies that, although question (Q1) stated in
the introduction is false, its relativized version to any analytic subset of
$\mathrm{SD}$ is true. Specifically, we have the following.
\begin{prop}
\label{p1} Let $A$ be an analytic class of separable dual spaces.
Let also $B$ be an analytic subset of $\mathrm{SD}$. Then the set
$A_*(B)=\{ X\in B: \exists Y\in A \text{ with } X^*\cong Y\}$ is analytic.
\end{prop}
\begin{proof}
Arguing as in the proof of Theorem \ref{t1}, we find a Borel subset $Z$ of $\sd$
such that $B\subseteq Z$. Define $G\subseteq Z\times H$ by $(X,f)\in G$ if and
only if $f\in K_{X^*}$. Then $G$ is Borel. Let $f_n:Z\to H$ $(n\in\nn)$ be the
sequence of Borel selectors of $G$ obtained by Proposition \ref{t2}. Let
also $I_k$ $(k\in\nn)$ and $S$ be the relations defined in the proof of
Theorem \ref{t1}. Now observe that
\begin{eqnarray*}
X\in A_*(B) & \Leftrightarrow & (X\in B) \text{ and } \big[ \exists Y\in A \
\exists (y_n)\in C(2^\nn)^\nn \ \exists k\geq 1 \text{ with } \\
& & \big( Y,(y_n)\big)\in S \text{ and } \big( (y_n), X\big)\in I_k \big].
\end{eqnarray*}
Hence $A_*(B)$ is analytic, as desired.
\end{proof}
\begin{rem}
Related to Proposition \ref{p1}, the following question is open to us.
Let $\phi$ be a co-analytic rank on $\mathrm{SD}$. Let also $A$ be an
analytic class of separable dual spaces such that for every $Y\in A$
there exists $\xi_Y<\omega_1$ with $\sup\{ \phi(X): X^*\cong Y\}<\xi_Y$.
Is, in this case, the set $A_*=\{X\in\sbs:\exists Y\in A \text{ with }
X^*\cong Y\}$ analytic? If this is true, then the counterexample to question
(Q1), presented in the introduction, is (in a sense) unique. We notice that
if we further assume that $\sup\{ \xi_Y: Y\in A\}<\omega_1$,
then Proposition \ref{p1} implies that the answer is positive.
\end{rem}
For every Banach space $X$ denote by $\sz(X)$ the Szlenk index
of $X$ (see \cite{Sz}). Let $\xi$ be a countable ordinal
and consider the class
\[ \mathcal{S}_\xi=\big\{ X\in\sbs: \max\{ \sz(X), \sz(X^*)\}\leq \xi\big\}.\]
By Theorem \ref{t1} and Proposition \ref{p1} we have the following.
\begin{cor}
\label{c1} For every countable ordinal $\xi$ the class
$\sss_\xi$ is analytic.
\end{cor}
\begin{proof}
Let us fix a countable ordinal $\xi$. As in the proof of Theorem
\ref{t1}, consider the subset $\sd$ of $\sbs$ consisting of all
Banach spaces with separable dual. We set $B=\{X\in\sd: \sz(X)\leq \xi\}$
and $A=B\cap B^*$. Notice that
\[ A=\big\{ Y\in\sbs: \sz(Y)\leq \xi \text{ and } (\exists X\in\sbs
\text{ with } \sz(X)\leq \xi \text{ and } Y\cong X^*)\big\}. \]
By \cite[Theorem 4.11]{Bos}, the map $X\mapsto \sz(X)$ is a co-analytic
rank on $\sd$. It follows that the set $B$ is analytic (in fact
Borel -- see \cite{Kechris}).  By Theorem \ref{t1}, so is the set $A$.
By Proposition \ref{p1}, we see that the set $A_*(B)$ is analytic.
As $A_*(B)=\sss_\xi$, the result follows.
\end{proof}
Let $\refl$ be the subset of $\sd$ consisting of all
separable reflexive spaces. Recently, E. Odell, Th.
Schlumprecht and A. Zs\a'{a}k have shown \cite[Theorem D]{OSZ}
that for every countable ordinal $\xi$ the class
\[ \ccc_\xi=\big\{ X\in\refl: \max\{\sz(X),\sz(X^*)\}\leq\xi\big\}\]
is also analytic. Their proof is based on Corollary \ref{c1} above, as well as, on
a deep refinement of M. Zippin's embedding theorem \cite{Z} and on a sharp
universality result concerning the classes
$\{\ccc_{\omega^{\xi\cdot\omega}}:\xi<\omega_1\}$
(Theorem B and Theorem C respectively in \cite{OSZ}).


\end{document}